\documentclass[11pt]{article}
 \usepackage{amsfonts}
 \usepackage{mathrsfs}
 \usepackage{bbm}
 \usepackage{amsthm}
\usepackage{amsmath,amssymb}
\newtheorem{theorem}{Theorem}[section]
\newtheorem{lemma}[theorem]{Lemma}

\theoremstyle{definition}

\theoremstyle{remark}
\newtheorem{remark}[theorem]{Remark}

\theoremstyle{algorithm}
\newtheorem{algorithm}[theorem]{Algorithm}

\theoremstyle{corollary}
\newtheorem{corollary}[theorem]{Corollary}

\theoremstyle{proposition}
\newtheorem{proposition}[theorem]{Proposition}

\theoremstyle{example}

\usepackage{amssymb}
\usepackage{epsfig}
\usepackage{times}
\title{On Equivalence and Computational Efficiency of the Major Relaxation Methods for  Minimum Ellipsoid Containing the Intersection of Ellipsoids}
\author{Zhiguo Wang,  Xiaojing Shen\thanks{This work was supported in part by  the open research funds of BACC-STAFDL of China
under Grant No. 2015afdl010, the NSFC No. 61673282 and the PCSIRT16R53. Zhiguo Wang, Xiaojing Shen (corresponding author), and Yunmin Zhu  are  with Department of Mathematics, Sichuan
University, Chengdu, Sichuan 610064, China. E-mail: wangzg315@126.com, shenxj@scu.edu.cn,  ymzhu@scu.edu.cn.} ~and Yunmin Zhu }

\begin{document}
 \maketitle
\begin{abstract}
This paper investigates the problem on the minimum ellipsoid containing the intersection of multiple ellipsoids, which has been extensively applied to information science, target tracking and data fusion etc. There are three major relaxation methods involving SDP relaxation, S-procedure relaxation and bounding ellipsoid relaxation, which are derived by different ideas or viewpoints. However, it is unclear for the interrelationships among these methods. This paper reveals the equivalence among the three relaxation methods by three stages. Firstly, the SDP relaxation method can be equivalently simplified to a decoupled SDP relaxation method. Secondly, the equivalence between the SDP relaxation method and the S-procedure relaxation method can be obtained by rigorous analysis. Thirdly, we establish the equivalence between the decoupled SDP relaxation method and the bounding ellipsoid relaxation method. Therefore, the three relaxation methods are unified through the decoupled SDP relaxation method. By analysis of the computational complexity, the decoupled SDP relaxation method has the least computational burden among the three methods. The above results are helpful for the research of set-membership filter and distributed estimation fusion. Finally, the performance of each method is evaluated by some typical numerical examples in information fusion and filtering.
\end{abstract}

\noindent{\bf keywords:} Semi-definite programming; the intersection of ellipsoids; S-procedure; optimal bounding ellipsoid; distributed estimation fusion.

\section{Introduction}
Problems involving the intersection of ellipsoids occur in many practiced fields such as information science, data fusion, automatic control and target tracking \cite{Boyd-ElGhaoui-Feron-Balakrishnan94,Wang-Li12,Shi-Chen-Lin15,Schweppe68,Shen-Zhu-Song-Luo11}. It is well known that the problems of state and parameter estimation are usually solved by a stochastic method with the noise assumed random \cite{Sun-Tian-Lin17}. However, the information of precise system or the distribution of noise may not be obtained. Thus, uncertain systems are alternatively considered in \cite{Polyak-Nazin-Durieu-Walter04,Kwon-Lee-Park11}, such as biased dynamic systems or unknown noise distributions \cite{Aubry-Boutayeb-Darouach08}. Since only requiring system bias and noise bounds are much easier than a precise system, it is sometimes more likely to assume that only bounds on the possible amplitude of the perturbations. Moreover, it is natural to use an ellipsoid to describe the uncertain set containing the all possible value of the unknown state vector or parameter \cite{Schweppe68}.
%The reason is that the ellipsoid has more degree of freedom.
Assume each estimated ellipsoid guarantees to contain the true state, then the intersection of these ellipsoids can also contain it. In fact, we want to find a minimum ellipsoid to contain the intersection of these ellipsoids.

The ellipsoidal bounding technique in the problem of state estimation with norm-bounded disturbance has attracted the attention of many researchers \cite{Joachim-Deller06}, \cite{Durieu-Walter-Polyak01}, \cite{Calafiore-ElGhaoui04}. For example, in \cite{Durieu-Walter-Polyak01} the authors consider two measures of the size of a minimum ellipsoid involving the approximation of the intersection of ellipsoids. In \cite{Ros-Sabater-Thomas02}, the authors also provide a suboptimal solution to the problem of finding the ellipsoid with minimum volume containing the intersection of the two ellipsoids defining the convex combination. In distributed fusion setting \cite{Duan-Li11,Yang-Liang-Pan-Qin-Yang16,Yuan-Wang-Guo17,Guo-Chen-Leung-Liu10}, when the cross-correlation of local sensor estimation errors is unknown or impractical, the covariance intersection (CI) algorithm are derived to deal with this problem in \cite{Uhlmann96,Julier-Uhlmann09,Cong-Li-Qi-Sheng16}. It provides not only a fused estimate point but also estimate covariance, the results are parameterized as convex combination of the local estimates.

For which the data is not specified exactly and only known to belong to a given uncertainty set, the authors in \cite{Ben-Tal-Nemirovski98} lay the foundation of robust convex optimization. It is argued that the more reasonable types of uncertainty sets are ellipsoids and intersection of finitely many ellipsoids. When we want to analyze the level of conservativeness of convex approximations to robust counterparts of semi-definite programming, the problem of ellipsoidal approximation of the intersection of ellipsoids is very useful \cite{Henrion-Tarbouriech-Arzelier01}. Although the problem of finding an optimal ellipsoid containing the intersection of ellipsoids exits in many engineering fields, unfortunately, so far, it has not been received enough attention.

%%% Problem Hardness
In fact, it is a nonconvex optimization problem for determining the minimum ellipsoid containing the intersection of several ellipsoids, since we need to maximize a convex function over a convex set. In \cite{Ben-Tal-Nemirovski98}, the authors  demonstrate that a robust optimization problem with uncertainty set such as the intersection of the ellipsoids %\textcolor[rgb]{1.00,0.00,0.00}{ a general-type ellipsoidal the intersection of the ellipsoids} \textcolor[rgb]{0.00,0.07,1.00}{uncertainty}
is  a semi-infinite optimization problem, which leads to computationally intractable robust counterpart.  The problem of determining a minimum ellipsoid to contain the intersection of the ellipsoids also involves solving nonconvex  quadratically constrained quadratic programs (QCQP) \cite{Palomar-Eldar10}, it captures many problems that are of interest to the signal processing, for instance, Boolean quadratic program (BQP). The BQP is long-known to be a computationally difficult problem, and it belongs to the class of NP-hard problems  \cite{Luo-Ma-So-Ye-Zhang10}. Moreover, the studied problem can be viewed as a special instance of calculating the minimum ellipsoid that covers a bounded convex set. In \cite{Boyd-Vandenberghe04} the authors formulate the problem of finding the minimum volume ellipsoid that covers a bounded convex set as convex programming problems, but they are tractable only in special cases, such as finite set, polyhedron, and union of ellipsoids. Therefore, we do not expect to solve the problem of determining the minimum ellipsoid containing the intersection of ellipsoids in polynomial time, or find the global solution of this problem.

%%%% Some Approximate Solutions
There are many authors study different methods for the problems involving the intersection of ellipsoids, and they may provide many approximate solutions by different ideas or viewpoints. In \cite{Boyd-ElGhaoui-Feron-Balakrishnan94}, the authors use S-procedure to approximate the
original problem to a semi-definite programming (SDP), then we can use the efficient interior point methods \cite{Nesterov-Nemirovski94} to obtain the solution.
 Since the ellipsoidal approximation of the intersection of the ellipsoids involves solving nonconvex QCQP, then the semi-definite relaxation (SDR) technique \cite{Luo-Ma-So-Ye-Zhang10} can be used. If these intersecting ellipsoids have a common center, the problem of quadratic form maximization over the intersection of these ellipsoids is studied in \cite{Nemirovski-Roos-Terlaky99}, and it also derives a new approximation bound based on the SDR technique. In \cite{Durieu-Walter-Polyak01} and \cite{Ros-Sabater-Thomas02}, the authors use the bounding ellipsoid relaxation method to contain the intersection of ellipsoids. Although these relaxation methods can obtain some ellipsoids to contain the intersection, it is unclear for the interrelationships among these methods and insight to the pros and cons of these methods.
In \cite{Henrion-Tarbouriech-Arzelier01}, the authors survey various linear matrix inequality relaxation techniques for evaluating the maximum norm vector within the intersection of several ellipsoids, however, which is different from the problem of  minimum ellipsoid containing the intersection of multiple ellipsoids. In fact,  the described problem in  \cite{Henrion-Tarbouriech-Arzelier01}  is similar to the Chebyshev center of a convex set \cite{Eldar-Beck-Teboulle08}.

%%%%% Our contribution
Motivated by the aforementioned analysis, the goal of this paper is to study the problem of determining a minimum ellipsoid containing the intersection of ellipsoids, and establish the equivalence of the three major relaxation methods involving SDP relaxation, S-procedure relaxation and bounding ellipsoid relaxation, meanwhile, we also analyze the computational efficiency of each method. Our contribution has three aspects:
\begin{itemize}
  \item Firstly, we present a comprehensive overview of some common relaxation techniques to deal with the problem of calculating the minimum ellipsoid containing the intersection of multiple ellipsoids, and the pros and cons among these relaxation methods are discussed.
   %To our knowledge, this is the first time to survey this important optimization problem in different setting.
  \item Secondly, we establish the equivalence among these relaxation methods though the decoupled SDP relaxation method, which has the least computational complexity among the three relaxation methods. The equivalence result is helpful for the research of set-membership filter and distributed estimation fusion.
  \item Thirdly, we derive an analytic expression of the shape matrix of the minimum ellipsoid by the decoupled SDP relaxation method. It is similar to that based on the covariance intersection method. However, we show that this minimum ellipsoid based on the decoupled SDP relaxation method is tighter than that based on the covariance intersection method.
\end{itemize}

The structure of the paper is as follows. In Section \ref{sec_2}, we describe the optimization problem on the minimum ellipsoid containing the intersection of multiple ellipsoids.  Section \ref{sec_3} presents the SDP relaxation method and the decoupled SDP relaxation method. Section \ref{sec_4} provides the S-procedure relaxation method. The equivalence between the SDP relaxation and S-procedure is established. Section \ref{sec_5} presents the bounding ellipsoid relaxation method. The equivalence between the decoupled SDP relaxation method and the bounding ellipsoid relaxation method is proved.  Simulations are given in Section \ref{sec_6}. Several conclusions and potential directions for further research are drawn in Section \ref{sec_7}.

\section{Problem Statement}\label{sec_2}
In many practical fields, we often need to deal with the problem of determining the minimum ellipsoid containing the intersection of multiple ellipsoids. For example, the measurement update step of set-membership filter always involves the intersection of predicted ellipsoid and measurement ellipsoid \cite{Schweppe68}. In the multisensor estimation fusion setting, each sensor sends the local estimated ellipsoid to the fusion center and cross-correlation is unknown, then we hope to derive an optimal ellipsoid to contain the intersection of local ellipsoids \cite{Noack-Sijs-Reinhardt-Hanebeck17}, which can improve the accuracy of estimation.

%\begin{figure}[!t]
%\centering
%\includegraphics[width=3in]{illustration.eps}
%\caption{Illustration of optimization problem on determining the optimal ellipsoid (E0) containing the intersection of ellipsoids (E1 and E2).}
%\label{fig8}
%\end{figure}

In order to derive some highly effective algorithms to solve this important problem, we describe it as the following optimization problem.
%Consider the optimization problem of computing the minimum ellipsoid containing the intersection of multiple ellipsoids
\begin{eqnarray}
\label{Eqpre_1} &&\min_{P_0,x_0}~~ f(P_0) \\
\label{Eqpre_2} &&\mbox{s.t.}~~ (x-x_0)^TP_0^{-1}(x-x_0)\leq1,\qquad\forall x\in\mathscr{F},
%\nonumber &&\qquad\qquad\forall x\in\mathscr{F},
\end{eqnarray}
 where the vector $x_0\in \mathcal{R}^n$ and the symmetric positive-definite matrix $P_0\in\mathcal{R}^{n\times n}$ are the decision variables, and they are also the center and the shape matrix of the optimal ellipsoid $\mathcal{E}_0$, respectively. The set $\mathscr{F}$ denotes the intersection of $m$ ellipsoids, which is defined as follows
\begin{eqnarray}
\label{Eqpre_45}\mathscr{F}&=&\bigcap_{i=1}^m\mathcal{E}_i,\\
\label{Eqpre_46} \mathcal{E}_i&=&\{x\in \mathcal{R}^n:(x-x_i)^T P_i^{-1}(x-x_i)\leq1\},~i=1,\ldots,m,
\end{eqnarray}
where $x_i$ and $P_i$ are the known center and the shape matrix of ellipsoid $\mathcal{E}_i$, respectively, and $P_i$ is a symmetric positive-definite matrix, $i=1,\ldots,m$. The objective function $f(P_0)$ is the ``size" of the optimized ellipsoid, which is a convex function and aimed at minimizing the shape matrix $P_0$. The common ``size" of the ellipsoid is $trace(P_0)$ or $logdet(P_0)$, which means the sum of squares of semiaxes lengths or the volume of the ellipsoid $\mathcal{E}_0$, respectively. For the two criterions, in \cite{Durieu-Walter-Polyak01}, the authors have proved this optimal ellipsoid exists and is unique. Actually, the constraint (\ref{Eqpre_2}) has infinite inequalities, which enforces that the intersection of these ellipsoids is contained in the optimized ellipsoid. Thus, the optimization problem (\ref{Eqpre_1})-(\ref{Eqpre_2}) is also called semi-infinite optimization problem \cite{Boyd-Vandenberghe04}.

Some authors would like to describe a general ellipsoid $\mathcal{E}_i$ as
\begin{eqnarray}
\label{Eqpre_49} \mathcal{E}_i&=&\{x:||A_ix-b_i||\leq1,~i=0,1,\ldots,m\},
\end{eqnarray}
i.e., the inverse image of the Euclidean unit ball under an affine mapping, where $A_i^TA_i=P_i^{-1}$, $b_i=A_ix_i$. Then the optimization problem (\ref{Eqpre_1})-(\ref{Eqpre_2}) is equivalent to the following problem in \cite{Boyd-Vandenberghe04},
\begin{eqnarray}
\label{Eqpre_47} &&\min_{A_0,b_0}~~ f(A_0^{-1}) \\
\label{Eqpre_48} &&~~\mbox{s.t.}~~ ||A_0x-b_0||\leq1,\qquad\forall x\in \mathscr{F},
%\nonumber && \qquad\qquad\forall x\in \mathscr{F},
\end{eqnarray}
where the optimization variables are the positive-definite matrix $A_0$ and the vector $b_0$.

In fact, an alternative description of the feasible set in (\ref{Eqpre_2}) is given by
\begin{eqnarray}
\label{Eqpre_3} \varphi(x_0,P_0)\leq0
\end{eqnarray}
with the function
\begin{eqnarray}
\label{Eqpre_4}\varphi(x_0,P_0)\triangleq \max_{x\in\mathscr{F}}~~(x-x_0)^TP_0^{-1}(x-x_0)-1.
%\label{Eqpre_6}&&~~\mbox{s.t.}~~~\forall x\in\mathscr{F}.
\end{eqnarray}
Obtaining the optimal value $\varphi(x_0,P_0)$ is very hard, since the objective function in (\ref{Eqpre_4}) is a convex in $x$, when $P_0$ is a positive definite matrix. In other words, it needs to calculate the maximum of a convex function, which is a non-convex optimization problem \cite{Horst-Tuy96}. Therefore, the computational complexity of the optimization problem  (\ref{Eqpre_1})-(\ref{Eqpre_2}) is NP-hard. Meanwhile, simply verifying that $\mathcal{E}_0\supset\bigcap_{i=1}^m\mathcal{E}_i$ holds, given $\mathcal{E}_1,\ldots,\mathcal{E}_m$, is NP-complete \cite{Boyd-ElGhaoui-Feron-Balakrishnan94}.

In next sections, we do not expect to solve the problem (\ref{Eqpre_1})-(\ref{Eqpre_2}) to obtain the optimized ellipsoid in polynomial time, but the efficient relaxation methods can be provided for approximation solutions. We concentrate on deriving the equivalence of typical relaxation methods and their interrelationships. Meanwhile, we also analyze the computational complexity of each method.
\section{SDP Relaxation}\label{sec_3}
%In this section, we derive the three classes of relaxation techniques for the optimization problem (\ref{Eqpre_1})-(\ref{Eqpre_2}) and show that some of them are actually equivalent. Meanwhile, these approximate optimization problems can be efficiently solved by interior method.
In this section, the optimization problem (\ref{Eqpre_1})-(\ref{Eqpre_2}) can be relaxed to an SDP problem by extending the relaxed Chebyshev center algorithm \cite{Eldar-Beck-Teboulle08}. Moreover, we derive an analytic expression of the shape matrix of the minimum ellipsoid based on a decoupled technique, and the SDP problem can be simplified to another convex problem, which has lower computational complexity.

\subsection{Convex Relaxation of $X=xx^T$}
Denoting $X=xx^T$, (\ref{Eqpre_4}) can be written equivalently as
\begin{eqnarray}
\label{Eqpre_7}&&\max~~f_0(X,x)\\
\label{Eqpre_8}&&~~\mbox{s.t.}~~~(X,x)\in\mathscr{F},
\end{eqnarray}
where
\begin{eqnarray}
\label{Eqpre_9}\mathscr{F}=\{(X,x):~f_i(X,x)\leq 0,~1\leq i\leq m,~X=xx^T\},
\end{eqnarray}
and
\begin{eqnarray}
\label{Eqpre_10}f_i(X,x)=tr(P_i^{-1}X)-2x_i^TP_i^{-1}x+x_i^TP_i^{-1}x_i-1,~0\leq i \leq m.
\end{eqnarray}
The objective function (\ref{Eqpre_7}) is linear in decision variable $(X,x)$, but the set $\mathscr{F}$ is not convex due to $X=xx^T$. In order to obtain a relaxation of the optimization problem (\ref{Eqpre_7})-(\ref{Eqpre_8}), usually $\mathscr{F}$ is replaced by the following convex set
\begin{eqnarray}
\label{Eqpre_11}\mathscr{T}=\{(X,x):~f_i(X,x)\leq0,~1\leq i\leq m,~xx^T\preceq X\}.
\end{eqnarray}

Based on the relaxed convex set $\mathscr{T}$, we can obtain a suboptimal ellipsoid to contain the intersection $\mathscr{F}$ by the SDP relaxation method as follows.
\begin{lemma}\label{the_1} (SDP relaxation)  \cite{Boyd-Vandenberghe04}:
The optimization problem (\ref{Eqpre_1})-(\ref{Eqpre_2}) can be relaxed to an SDP problem as follows
\begin{eqnarray}
\label{Eqpre_28} &&\min~~ f(P_0) \\
%\nonumber&&s.t. \left[
%                  \begin{array}{ccc}
%                    P_0^{-1}-\sum_{i=1}^m\lambda_iP_i^{-1} & -\tilde{x}_0+\sum_{i=1}^m\lambda_iP_i^{-1}x_i & 0 \\
%                    (-\tilde{x}_0+\sum_{i=1}^m\lambda_iP_i^{-1}x_i)^T & -1-\sum_{i=1}^m\lambda_i(x_i^TP_i^{-1}x_i-1)& \tilde{x}_0^T \\
%                    0 & \tilde{x}_0 & -P_0^{-1} \\
%                  \end{array}
%                \right]\\
\label{Eqpre_19}&&\mbox{s.t.}~\lambda_i\geq0,i=1,\ldots,m,\qquad\qquad\qquad\qquad\qquad\qquad\\
\label{Eqpre_29}
&&\left[
                  \begin{array}{ccc}
                    P_0^{-1}-\sum_{i=1}^m\lambda_iP_i^{-1} & -\tilde{x}_0+\sum_{i=1}^m\lambda_iP_i^{-1}x_i & 0 \\
                    (-\tilde{x}_0+\sum_{i=1}^m\lambda_iP_i^{-1}x_i)^T & -1-\sum_{i=1}^m\lambda_i(x_i^TP_i^{-1}x_i-1)& \tilde{x}_0^T \\
                    0 & \tilde{x}_0 & -P_0^{-1} \\
                  \end{array}
                \right]\\
 \nonumber&&\qquad \qquad\qquad\qquad \qquad\qquad\qquad\qquad\qquad \qquad\qquad\qquad\qquad            \preceq0,
\end{eqnarray}
where $\tilde{x}_0=P_0^{-1}x_0$. Moreover, the relaxed optimization problem (\ref{Eqpre_28})-(\ref{Eqpre_29}) is convex in the variables $P_0^{-1}$, $\tilde{x}_0$, $\lambda_1,\ldots,\lambda_m$.
\end{lemma}
\begin{proof}\label{pro_1} %See Appendix.
In order to see clearly the equivalence of these relaxation methods, we present a proof for our specific formulation
of Lemma \ref{the_1} in Appendix.
\end{proof}

 \begin{algorithm}\label{alg_2} SDP relaxation
 \begin{enumerate}
   \item Solve the optimization problem (\ref{Eqpre_28})-(\ref{Eqpre_29}) for the optimal solutions $P_0^{-1}$ and $\tilde{x}_0$.
   \item Compute the shape matrix $P_0$ and the center $x_0=P_0\tilde{x}_0$.
 \end{enumerate}
 \end{algorithm}
As we know, this SDP problem (\ref{Eqpre_28})-(\ref{Eqpre_29}) can be efficiently solved by interior point methods \cite{Vandenberghe-Boyd96}. Meanwhile, we can use the convex optimization toolbox CVX \cite{Grant-Boyd-Ye09} to solve (\ref{Eqpre_28})-(\ref{Eqpre_29})  in MATLAB, and Algorithm \ref{alg_2} gives us a suboptimal ellipsoid that contains the intersection of ellipsoids $\mathcal{E}_1,\ldots,\mathcal{E}_m$.

\subsection{Decoupled SDP relaxation}\label{sec_3_2}
In this subsection, we use a decoupled technique \cite{Calafiore-ElGhaoui04} to make further efforts to improve the complexity of this SDP problem (\ref{Eqpre_28})-(\ref{Eqpre_29}). The specific method can be seen in the following proposition.
\begin{proposition} \label{cor_2}
When the objective function is trace or logdet function, the optimization problem (\ref{Eqpre_28})-(\ref{Eqpre_29}) based on SDP relaxation technique can be equivalently decoupled to
\begin{eqnarray}
\label{Eqpre_50} &&\min~~ f\left(\Big(\sum_{i=1}^m \lambda_iP_i^{-1}\Big)^{-1}\right) \\
 \nonumber&&\mbox{s.t.}~~ \lambda_i\geq0,~i=1,\ldots,m,\\
 \label{Eqpre_51}&& \left[
                                     \begin{array}{cc}
                                       1-\sum_{i=1}^m \lambda_i+\sum_{i=1}^m \lambda_ix_i^TP_i^{-1}x_i&\sum_{i=1}^m \lambda_ix_i^TP_i^{-1} \\
                                      \sum_{i=1}^m \lambda_iP_i^{-1}x_i & \sum_{i=1}^m \lambda_iP_i^{-1} \\
                                     \end{array}
                                   \right]\succeq0,
 %\label{Eqpre_51}&&                  \left[
%                                     \begin{array}{cc}
%                                       1-\sum_{i=1}^m \lambda_i+ & ~ \\
%                                       \sum_{i=1}^m \lambda_ix_i^TP_i^{-1}x_i & \sum_{i=1}^m \lambda_ix_i^TP_i^{-1}\\
%                                       ~ &~ \\
%                                       \sum_{i=1}^m \lambda_iP_i^{-1}x_i & \sum_{i=1}^m \lambda_iP_i^{-1}\\
%                                     \end{array}
%                                   \right]\succeq0
%\\
%\nonumber &&\qquad \qquad \lambda_i\geq0,~i=1,\ldots,m,
\end{eqnarray}
where $\lambda_i,~i=1,\ldots,m$ are the optimization variables of this problem.
If the above problem is feasible, then there exists an optimal solution for optimization problem (\ref{Eqpre_28})-(\ref{Eqpre_29}). In this case, calling $\lambda_i^{*}$ optimal values of the problem variable $\lambda_i$, $i=1,\ldots,m$, then the optimal shape matrix and the center of problem (\ref{Eqpre_28})-(\ref{Eqpre_29}) satisfy
\begin{eqnarray}
\label{Eqpre_52} P_0^{-1}&=&\sum_{i=1}^m \lambda_i^{*}P_i^{-1}\\
\label{Eqpre_53} x_0&=&P_0\sum_{i=1}^m \lambda_i^{*}P_i^{-1}x_i.
\end{eqnarray}
\end{proposition}

\begin{proof}
%This result is inspired by the decoupled technique in \cite{ElGhaoui-Calafiore01}. Next,
%We present a proof for the specific formulation of the proposition.
Using Schur complement, (\ref{Eqpre_29}) is feasible if and only if
\begin{eqnarray}
\label{Eqpre_55}\left[
  \begin{array}{cc}
   \sum_{i=1}^m\lambda_iP_i^{-1}- P_0^{-1} & \tilde{x}_0-\sum_{i=1}^m\lambda_iP_i^{-1}x_i \\
    (\tilde{x}_0-\sum_{i=1}^m\lambda_iP_i^{-1}x_i)^T & 1+\sum_{i=1}^m\lambda_i(x_i^TP_i^{-1}x_i-1)-\tilde{x}_0^TP_0\tilde{x}_0 \\
  \end{array}
\right]\succeq0.
\end{eqnarray}

where $\tilde{x}_0=P_0^{-1}x_0$. Thus, we only need to prove that (\ref{Eqpre_55}) is feasible if and only if (\ref{Eqpre_51}) is feasible.
It is given from the following two aspects:

 $``\Longleftarrow"$
Based on the definition of the shape matrix $P_0$ and the center $x_0$ in (\ref{Eqpre_52})-(\ref{Eqpre_53}),
\begin{eqnarray}
\nonumber P_0^{-1}&=&\sum_{i=1}^m \lambda_iP_i^{-1}\\
\nonumber  x_0&=&P_0\sum_{i=1}^m \lambda_iP_i^{-1}x_i.
\end{eqnarray}
If the optimal values $\lambda_i$, $i=1,\ldots,m$, satisfy the constraint (\ref{Eqpre_51}), then the optimal variables $P_0, x_0,\lambda_i$, $i=1,\ldots,m$ are also feasible for the constraint (\ref{Eqpre_55}) by Schur complement.
That is, that (\ref{Eqpre_51}) is feasible implies (\ref{Eqpre_55}) is feasible.

$``\Longrightarrow"$ Based on Schur complement, (\ref{Eqpre_55}) is equivalent to
\begin{eqnarray}
\nonumber &&\sum_{i=1}^m\lambda_iP_i^{-1}-P_0^{-1}-\Big(\tilde{x}_0-\sum_{i=1}^m\lambda_iP_i^{-1}x_i\Big)\\
 \nonumber &&\cdot\Big(1+\sum_{i=1}^m\lambda_i(x_i^TP_i^{-1}x_i-1)-\tilde{x}_0^TP_0\tilde{x}_0\Big)^{-1}\\
 \label{Eqpre_56}&&\cdot\Big(\tilde{x}_0-\sum_{i=1}^m\lambda_iP_i^{-1}x_i\Big)^T
  \succeq0.
\end{eqnarray}
Both in the case of trace and logdet function, $f(X_1)\geq f(X_2)$ whenever $X_1\succeq X_2$. Then, according to  the decoupled technique in \cite{Calafiore-ElGhaoui04}, the minimum of $f(P_0)$ is achieved, when
$P_0=\Big(\sum_{i=1}^m \lambda_iP_i^{-1}\Big)^{-1}$ and $\tilde{x}_0=\sum_{i=1}^m\lambda_iP_i^{-1}x_i$. Thus, (\ref{Eqpre_55}) is feasible implies (\ref{Eqpre_51}) is feasible.

All in all, the problem (\ref{Eqpre_50})-(\ref{Eqpre_51}) is equivalent to the optimization problem (\ref{Eqpre_28})-(\ref{Eqpre_29}).
\end{proof}

  \begin{algorithm}\label{alg_3} Decoupled SDP relaxation
 \begin{enumerate}
   \item Solve the optimization problem (\ref{Eqpre_50})-(\ref{Eqpre_51}) for the optimal solutions $\lambda_i$, $i=1,\ldots,m$.
   \item Compute the shape matrix $P_0$ and the center $x_0$ by (\ref{Eqpre_52})-(\ref{Eqpre_53}).
 \end{enumerate}
 \end{algorithm}

\begin{remark}\label{rem_1}~
\begin{itemize}
  \item  Proposition \ref{cor_2} shows that the shape matrix $P_0$ and the center $x_0$ (\ref{Eqpre_52})-(\ref{Eqpre_53}) of the minimum ellipsoid derived by the decoupled relaxation method are the weighted combination of those of multiple ellipsoids. The optimization problem (\ref{Eqpre_50})-(\ref{Eqpre_53}) is also a significant bridge for deriving the equivalence among three typical relaxation methods in next section.
  \item  From Proposition \ref{cor_2}, we can calculate a suboptimal ellipsoid to contain the intersection $\mathscr{F}$ by (\ref{Eqpre_50})-(\ref{Eqpre_53}), which is described by Algorithm \ref{alg_3}.
       Comparing the optimization problem (\ref{Eqpre_50})-(\ref{Eqpre_51}) with (\ref{Eqpre_28})-(\ref{Eqpre_29}), it is easy to see that the dimension of the constraint matrix (\ref{Eqpre_51}) of Algorithm \ref{alg_3} is $n+1$, which is smaller than that of matrix (\ref{Eqpre_29}). Meanwhile, the number of the decision variables of the optimization problem (\ref{Eqpre_50})-(\ref{Eqpre_51}) is $m$, but that of the optimization problem (\ref{Eqpre_28})-(\ref{Eqpre_29}) is $m+\frac{n(n+1)}{2}+n$. Therefore, if we use Proposition \ref{cor_2} to calculate the ellipsoid containing the intersection of the ellipsoids, the computing time may be reduced much more, which can be clearly seen in Table \ref{tab_1} in the numerical section.
\end{itemize}

\end{remark}
\section{S-Procedure Relaxation} \label{sec_4}
In system and control theory, one often encounters the constraint that a quadratic form is negative when other quadratic forms are all negative. In some cases, this constraint can be relaxed as a linear matrix inequality (LMI) based on S-procedure method. Actually, the ellipsoids can be written as quadratic forms in (\ref{Eqpre_49}), therefore, \cite{Boyd-ElGhaoui-Feron-Balakrishnan94} uses the S-procedure relaxation technique to deal with the constrain condition (\ref{Eqpre_2}).
\subsection{The Equivalence Between S-Procedure Relaxation and SDP Relaxation}
\begin{lemma}\label{lem_1} (S-procedure)\cite{Yakubovich71}:
Let $F_0(\xi), F_1(\xi),\ldots, F_m(\xi)$, be quadratic functions in variable $\xi\in\mathcal {R}^{n}$
\begin{eqnarray}
F_i(\xi)=\xi^TT_i\xi, ~~i=0,\ldots, m
\end{eqnarray}
with $T_i=T_i^T$. Then the implication
\begin{eqnarray}
F_1(\xi)\leq0,\ldots,F_m(\xi)\leq0\Rightarrow F_0(\xi)\leq0
\end{eqnarray}
holds if there exist $\tau_1,\ldots,\tau_m\geq0$ such that
\begin{eqnarray}
F_0(\xi)-\sum_{i=1}^m\tau_iF_i(\xi)\leq0, ~ for~ all ~\xi,
\end{eqnarray}
or
\begin{eqnarray}
T_0-\sum_{i=1}^m\tau_iT_i\preceq0.
\end{eqnarray}
It is a nontrivial fact that when $m= 1$, the converse holds \cite{Boyd-ElGhaoui-Feron-Balakrishnan94}.
\end{lemma}
Now, a suboptimal ellipsoid for the intersection $\mathscr{F}$ can be derived by the S-procedure relaxation technique.
Denoting $\xi=[x^T~1]^T$, then (\ref{Eqpre_46}) can be equivalent to
\begin{eqnarray}
\label{Eqpre_22} f_i(X,x)=\xi^TA_i\xi\leq0, ~~0\leq i \leq m,
\end{eqnarray}
where
\begin{eqnarray}
\label{Eqpre_23} A_i=\left[
                         \begin{array}{cc}
                           P_i^{-1} & -P_i^{-1}x_i \\
                           -x_i^TP_i^{-1} & x_i^TP_i^{-1}x_i-1 \\
                         \end{array}
                       \right]
, ~~0\leq i \leq m.
\end{eqnarray}
Thus, the problem $\mathcal{E}_0\supset\bigcap_{i=1}^m\mathcal{E}_i$ is equivalent to
\begin{eqnarray}
\label{Eqpre_24} f_0(X,x)=\xi^TA_0\xi\leq0
\end{eqnarray}
whenever
\begin{eqnarray}
\label{Eqpre_25} f_i(X,x)=\xi^TA_i\xi\leq0,~~1\leq i \leq m.
\end{eqnarray}
By S-procedure, we can obtain $\mathcal{E}_0$ containing $\bigcap_{i=1}^m\mathcal{E}_i$ if there exist $\tau_1,\ldots,\tau_m\geq0$ such that
\begin{eqnarray}
\label{Eqpre_26} A_0\preceq\sum_{i=1}^m\tau_iA_i.
\end{eqnarray}
Using the definition of $A_i$ in (\ref{Eqpre_23}), so we can write the equation (\ref{Eqpre_26}) as an LMI
\begin{eqnarray}
\label{Eqpre_27}&&\left[
                         \begin{array}{cc}
                           P_0^{-1} & -P_0^{-1}x_0 \\
                           -x_0^TP_0^{-1} & x_0^TP_0^{-1}x_0-1 \\
                         \end{array}
                       \right]\preceq\sum_{i=1}^m\tau_i\left[
                         \begin{array}{cc}
                           P_i^{-1} & -P_i^{-1}x_i \\
                           -x_i^TP_i^{-1} & x_i^TP_i^{-1}x_i-1 \\
                         \end{array}
                       \right].
\end{eqnarray}
Then, based on the S-procedure relaxation technique, the best such outer ellipsoid of the intersection $\mathscr{F}$ can be derived by solving the following optimization problem
\begin{eqnarray}
\label{Eqpre_76} &&\min~~ f(P_0) \\
\nonumber
&&\mbox{s.t.}
~\tau_i\geq0,i=1,\ldots,m,\\
\label{Eqpre_77}&& \left[
                         \begin{array}{cc}
                           P_0^{-1} & -P_0^{-1}x_0 \\
                           -x_0^TP_0^{-1} & x_0^TP_0^{-1}x_0-1 \\
                         \end{array}
                       \right]\preceq\sum_{i=1}^m\tau_i\left[
                         \begin{array}{cc}
                           P_i^{-1} & -P_i^{-1}x_i \\
                           -x_i^TP_i^{-1} & x_i^TP_i^{-1}x_i-1 \\
                         \end{array}
                       \right]
\end{eqnarray}
with variables $P_0$, $x_0$, $\tau_i$, $i=1,\ldots,m$.
Interestingly, replacing the variable $\tilde{x}_0$ by $\tilde{x}_0=P_0^{-1}x_0$, (\ref{Eqpre_77}) is equivalent to (\ref{Eqpre_29}) by Schur complement (see the proof of Lemma \ref{the_1}). Therefore, we can obtain the following proposition.

\begin{proposition} \label{pro_4}
The optimization problem (\ref{Eqpre_28})-(\ref{Eqpre_29}) based on the SDP relaxation technique is equivalent to the optimization problem (\ref{Eqpre_76})-(\ref{Eqpre_77}) by the S-procedure relaxation method, i.e., they can derive the same minimum ellipsoid to contain the intersection of these ellipsoids.
\end{proposition}

\subsection{S-Procedure Relaxation with Maximum Volume Inscribed Ellipsoid}
Another method that can derive a suboptimal solution for the original problem (\ref{Eqpre_1})-(\ref{Eqpre_2}) is to first calculate the maximum volume ellipsoid that is contained in the intersection, which can be also cast as convex problem with LMI constrains \cite{Boyd-Vandenberghe04}. If we use Lemma \ref{lem_3} to shrink this ellipsoid by a factor of $n$ about it center, then it is guaranteed to contain the intersection.

\begin{lemma}\label{lem_3}\cite{Boyd-Vandenberghe04}
Let $\mathcal{E}_{lj}$ be the L$\ddot{o}$wner-John ellipsoid \footnote{The minimum volume ellipsoid that contains a set $C$ is called the L$\ddot{o}$wner-John ellipsoid of the set $C$.} of the convex set $C\subseteq\mathcal{R}^n$, which is a  compact set with a nonempty  interior, and let $x_{lj}$ be its center. Then,
\begin{eqnarray}
\label{Eqpre_62} x_{lj}+\frac{1}{n}(\mathcal{E}_{lj}-x_{lj})\subseteq C \subseteq \mathcal{E}_{lj}.
\end{eqnarray}
\end{lemma}

Now, we consider the  problem of finding a maximum volume ellipsoid that lies inside the intersection  $\mathscr{F}$. Let ellipsoid $\bar{\mathcal{E}}_{0}$ is contained in the intersection of the ellipsoids $\mathcal{E}_{1},\ldots,\mathcal{E}_{m}$, and it is defined as follows
\begin{eqnarray}
\bar{\mathcal{E}}_{0}&=&\{x\in \mathcal{R}^n:(x-\bar{x}_0)^T \bar{P}_0^{-1}(x-\bar{x}_0)\leq1\}\\
                       &=&\{x\in \mathcal{R}^n: \xi^T\bar{A}_0\xi\leq0, \xi=[x^T~1]^T\},
\end{eqnarray}
where
\begin{eqnarray}
 \bar{A}_0=\left[
                         \begin{array}{cc}
                           \bar{P}_0^{-1} & -\bar{P}_0^{-1}\bar{x}_0 \\
                           -\bar{x}_0^T\bar{P}_0^{-1} & \bar{x}_0^T\bar{P}_0^{-1}\bar{x}_0-1 \\
                         \end{array}
                       \right].
\end{eqnarray}

Since $\bar{\mathcal{E}}_{0}\subseteq\mathcal{E}_{i}$ if and only if for every $x$ satisfying
 \begin{eqnarray}
 (x-\bar{x}_0)^T \bar{P}_0^{-1}(x-\bar{x}_0)\leq1,
 \end{eqnarray}
we have
 \begin{eqnarray}
 (x-x_i)^T P_i^{-1}(x-x_i)\leq1, ~i=1,\ldots,m.
 \end{eqnarray}
 Then, by Lemma \ref{lem_1}, this is equivalent to that there exist nonnegative scalars $\tau_1,\ldots,\tau_m$ satisfying
  \begin{eqnarray}
 \label{Eqpre_58}&&(x-x_i)^T P_i^{-1}(x-x_i)-\tau_i(x-\bar{x}_0)^T \bar{P}_0^{-1}(x-\bar{x}_0)\leq1-\tau_i,\\
 \nonumber &&\qquad\qquad\qquad~i=1,\ldots,m,~ \forall~x\in\mathcal{R}^n.
 \end{eqnarray}
 Using Schur Complements, \cite{Boyd-ElGhaoui-Feron-Balakrishnan94} has proved that (\ref{Eqpre_58}) is equivalent to the following LMI
  \begin{eqnarray}
   \label{Eqpre_59}&&\left[
                      \begin{array}{ccc}
                        -P_i & x_i-\bar{x}_0 & \bar{E}_0 \\
                        (x_i-\bar{x}_0)^T & \tau_i-1 & 0 \\
                        \bar{E}_0^T & 0 & -\tau_iI\\
                      \end{array}
                    \right]\preceq 0,~~i=1,\ldots,m,
   \end{eqnarray}
   where $\bar{P}_0=\bar{E}_0^T\bar{E}_0$, $I$ is identity matrix.
Therefore, we can obtain the optimal maximum volume ellipsoid contained in the intersection of the ellipsoids $\mathcal{E}_1,\ldots,\mathcal{E}_m$ by solving the problem
\begin{eqnarray}
\label{Eqpre_60} &&\max~~ \log \det(\bar{E}_0) \\
\label{Eqpre_61} &&~~\mbox{s.t.}~~ \tau_i\geq0, i=1,\ldots,m, (\ref{Eqpre_59})
\end{eqnarray}
with variables $\bar{E}_0$, $\bar{x}_0$, $\tau_i$, $i=1,\ldots,m$. Next, we can use Lemma $\ref{lem_3}$ to achieve an ellipsoid containing the intersection of the ellipsoids. Algorithm \ref{alg_4} gives us the specific process to calculate an outer ellipsoid, indeed, it is a rather conservative approximation \cite{Hanif-Tran-Antti-Glisic13}.

  \begin{algorithm}\label{alg_4} S-procedure relaxation with maximum volume inscribed ellipsoid
 \begin{enumerate}
   \item Solve the optimization problem (\ref{Eqpre_60})-(\ref{Eqpre_61}) of determining the maximum volume ellipsoid for the optimal solutions $\bar{E}_0$ and $\bar{x}_0$.
   \item Compute the shape matrix $P_0$ and the center $x_0$ of the outer ellipsoid  by (\ref{Eqpre_62}).
 \end{enumerate}
 \end{algorithm}
\section{Bounding Ellipsoid Relaxation}\label{sec_5}
The authors \cite{Durieu-Walter-Polyak01} divide the problem of computing the minimum bounding ellipsoid of the intersection of ellipsoids into two steps. First, one needs to find a parameterized ellipsoid to contain the intersection. Second, the optimal ellipsoid can be derived by optimizing the parameter.
\subsection{The Equivalence Between Bounding Ellipsoid Relaxation and Decoupled SDP Relaxation}
Assume that the intersection $\mathscr{F}$ of $m$ ellipsoids is a nonempty bounded region, which is defined by (\ref{Eqpre_45}), i.e.,
\begin{eqnarray}
\label{Eqpre_31} \mathscr{F}=\bigcap_{i=1}^m\{x:(x-x_i)^T P_i^{-1}(x-x_i)\leq1,~i=1,\ldots,m\}.
\end{eqnarray}
Let $\mathscr{D}^+$ be the convex set of all vectors $t=(t_1,\ldots,t_m)\in\mathcal{R}^m$, with $t_i\geq0$, $i=1,\ldots,m$, and $\sum_{i=1}^m t_i=1$.
If $x\in\mathscr{F}$, then
\begin{eqnarray}
\label{Eqpre_32} \sum_{i=1}^m t_i(x-x_i)^T P_i^{-1}(x-x_i)\leq1, ~\forall~t\in\mathscr{D}^+.
\end{eqnarray}
After simple calculations and transformations, we can obtain the following equation,
\begin{eqnarray}
 \nonumber&&\sum_{i=1}^m t_i(x-x_i)^T P_i^{-1}(x-x_i)\\
\label{Eqpre_33}&&=(x-x_t)^TP_t^{-1}(x-x_t)+\delta_t,
\end{eqnarray}
where

%it is coincides with $\mathcal{E}_i$ for $t_i=1$ and $t_1=\cdots=t_{i-1}=t_{i+1}=\cdots=t_m=0$.
%
%In order to find a tight ellipsoid to contain the set $\mathscr{F}$, we can translate (\ref{Eqpre_31}) as follows
%\begin{eqnarray}
%\label{Eqpre_32} \mathscr{F}=\Big\{x: (x-x_0)^TP_0^{-1}(x-x_0)\leq \delta, ~\forall t_i\geq0,~i=1,\ldots,m\Big\},
%\end{eqnarray}
%where
\begin{eqnarray}
\label{Eqpre_34} P_t^{-1}&=&\sum_{i=1}^m t_iP_i^{-1}\\
\label{Eqpre_35} x_t&=&P_t\sum_{i=1}^m t_iP_i^{-1}x_i\\
\label{Eqpre_36} \delta_t&=&\sum_{i=1}^m t_ix_i^TP_i^{-1}x_i-x_t^TP_t^{-1}x_t.
\end{eqnarray}
Then, $\mathscr{F}$ can be rewritten as follows \cite{Durieu-Walter-Polyak01,Ros-Sabater-Thomas02},
\begin{eqnarray}
\label{Eqpre_37}\mathscr{F}=\{x: (x-x_t)^TP_t^{-1}(x-x_t)\leq 1-\delta_t,~\forall t\in\mathscr{D}^+\}.
\end{eqnarray}
Obviously, (\ref{Eqpre_37}) may not be an ellipsoid. Since $\mathscr{F}$ is a nonempty set, then $ \delta_t\leq 1$. Let the ellipsoid $\mathcal{E}_t$ denoted as $\mathcal{E}_t=\{x: (x-x_t)^T((1-\delta_t)P_t)^{-1}(x-x_t)\leq 1\}$, then $\mathscr{F}\subseteq \mathcal{E}_t$. Therefore, we have found an ellipsoid to contain the intersection of multiple ellipsoids, but it depends on the parameter $t_i, i=1,\ldots,m$. In order to derive an optimal ellipsoid $\mathcal{E}_t$, we need to minimize the function of the shape matrix $(1-\delta_t)P_t$. The optimization problem can be written as follows
\begin{eqnarray}
\label{Eqpre_63} &&\min~~ f((1-\delta_t)P_t) \\
\label{Eqpre_64} &&~~\mbox{s.t.}~~ (t_1,\ldots,t_m)\in\mathscr{D}^+,
\end{eqnarray}
where the vector $t=(t_1,\ldots,t_m)$ is the optimization variable of this problem.

If we can obtain the optimal solution $t^{*}$ of the optimization problem (\ref{Eqpre_63})-(\ref{Eqpre_64}), then according to (\ref{Eqpre_34})-(\ref{Eqpre_35}), an outer suboptimal ellipsoid for intersection  $\mathscr{F}$ can be derived. In Section \ref{sec_3_2}, we also derive the analytic expressions (\ref{Eqpre_52})-(\ref{Eqpre_53}) of the shape matrix and the center of the minimum ellipsoid by the decoupled SDP relaxation method, interestingly, they are similar to the equations (\ref{Eqpre_34})-(\ref{Eqpre_35}). In fact, the optimization problems (\ref{Eqpre_63})-(\ref{Eqpre_64}) and (\ref{Eqpre_50})-(\ref{Eqpre_51}) are equivalent for calculating the weight of the each ellipsoid but not in form.
\begin{proposition}\label{pro_3}
When the objective function is logdet or trace function, then the optimization problem (\ref{Eqpre_63})-(\ref{Eqpre_64}) based on the bounding ellipsoid relaxation technique is equivalent to the optimization problem (\ref{Eqpre_50})-(\ref{Eqpre_51}) based on the decoupled SDP relaxation method, i.e., they have the same minimum ellipsoid to contain the intersection  $\mathscr{F}$.
\end{proposition}
\begin{proof}
 The proof has three steps. Firstly, when the objective function is logdet function, we continue to simplify the optimization problem (\ref{Eqpre_50})-(\ref{Eqpre_51}). Let $\eta=\sum_{i=1}^m\lambda_i$ and $t_i=\frac{\lambda_i}{\eta}$, $i=1,\ldots,m$, then $\sum_{i=1}^mt_i=1$. Since $\eta>0$ \cite{Boyd-ElGhaoui-Feron-Balakrishnan94}, if the constraint (\ref{Eqpre_51}) is scaled by a factor of $\frac{1}{\eta}$, then the optimization problem (\ref{Eqpre_50})-(\ref{Eqpre_51}) can be rewritten as
 \begin{eqnarray}
\label{Eqpre_65} &&\min~~ -\log\det\Big(\sum_{i=1}^m t_iP_i^{-1}\Big)-n\cdot \log(\eta) \\
\label{Eqpre_66} &&\mbox{s.t.}~
\left[
                                     \begin{array}{cc}
                                       \frac{1}{\eta}-\sum_{i=1}^m t_i+\sum_{i=1}^m t_ix_i^TP_i^{-1}x_i&\sum_{i=1}^m t_ix_i^TP_i^{-1} \\
                                      \sum_{i=1}^m t_iP_i^{-1}x_i & \sum_{i=1}^m t_iP_i^{-1} \\
                                     \end{array}
                                   \right]\succeq0\\
%                                   \left[
%                                     \begin{array}{cc}
%                                       \frac{1}{\eta}-\sum_{i=1}^m t_i+ & \sum_{i=1}^m t_ix_i^TP_i^{-1} \\
%                                       \sum_{i=1}^m t_ix_i^TP_i^{-1}x_i & ~ \\
%                                       ~ & ~ \\
%                                       \sum_{i=1}^m t_iP_i^{-1}x_i & \sum_{i=1}^m t_iP_i^{-1} \\
%                                     \end{array}
%                                   \right]\succeq0\\
  \label{Eqpre_67} &&~~~~ ~ \sum_{i=1}^mt_i=1,~ t_i\geq0,~i=1,\ldots,m.
\end{eqnarray}
Let $\xi=\frac{1}{\eta}$, and use Schur Complements, the above optimization problem (\ref{Eqpre_65})-(\ref{Eqpre_67}) is equivalent to the following problem
   \begin{eqnarray}
\label{Eqpre_68} &&\min~ -\log\det\Big(\sum_{i=1}^m t_iP_i^{-1}\Big)+n\cdot \log(\xi) \\
 \label{Eqpre_69} &&\mbox{s.t.}~  \xi-\sum_{i=1}^m t_i+\sum_{i=1}^m t_ix_i^TP_i^{-1}x_i\\
 \nonumber &&\qquad-\sum_{i=1}^m t_ix_i^TP_i^{-1}\Big(\sum_{i=1}^m t_iP_i^{-1}\Big)^{-1}\sum_{i=1}^m t_iP_i^{-1}x_i\geq0  \\
  \label{Eqpre_70} &&~  \sum_{i=1}^mt_i=1,~ t_i\geq0,~i=1,\ldots,m.
\end{eqnarray}
Secondly, because of the monotonicity of the $\log$ function, we can introduce an optimization variable $\xi$, then the optimization problem (\ref{Eqpre_63})-(\ref{Eqpre_64})
\begin{eqnarray}
\label{Eqpre_71} &&\min~~ n\cdot \log(1-\delta_t)-\log\det\Big(\sum_{i=1}^m t_iP_i^{-1}\Big) \\
\label{Eqpre_72} &&~~\mbox{s.t.}~~ (t_1,\ldots,t_m)\in\mathscr{D}^+
\end{eqnarray}
is equivalent to
\begin{eqnarray}
\label{Eqpre_73} &&\min~~ n\cdot \log(\xi)-\log\det\Big(\sum_{i=1}^m t_iP_i^{-1}\Big) \\
\label{Eqpre_74} &&~~ \mbox{s.t.}~~\xi-1+\delta_t\geq0,\\
\label{Eqpre_75} &&~~\qquad (t_1,\ldots,t_m)\in\mathscr{D}^+.
\end{eqnarray}
Thirdly, according to the definition of $\delta_t$ in (\ref{Eqpre_36}), it is shown that the optimization problems (\ref{Eqpre_73})-(\ref{Eqpre_75}) and (\ref{Eqpre_68})-(\ref{Eqpre_70}) have the same optimal solution. When the objective function is trace function, the proof is similar. Thus the proof of Proposition \ref{pro_3} is finished.
\end{proof}
\begin{remark}\label{rem_2}
Based on Proposition \ref{cor_2}, Proposition \ref{pro_4} and Proposition \ref{pro_3}, we can derive that the SDP relaxation method, S-procedure method and bounding ellipsoid relaxation method can achieve the same solution for the optimization problem (\ref{Eqpre_1})-(\ref{Eqpre_2}). The equivalence among these important relaxation techniques is showed in Fig. \ref{fig2}.
\end{remark}
%\begin{center}
%  \centerline
%  {\includegraphics[scale=0.32]{equ.eps}}\vskip3mm
%\centering{\small { Fig. 1}\ \ The equivalence among these relaxation techniques. \label{fig2}}
%\end{center}

\begin{figure}[!t]
\centering
\includegraphics[width=2.5in]{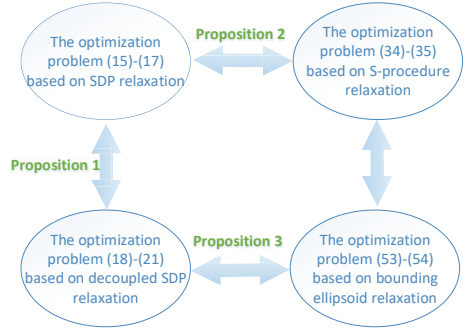}
\caption{The equivalence among these relaxation techniques.}
\label{fig2}
\end{figure}
\subsection{The Significance of the Equivalence}
Actually, in \cite{Durieu-Walter-Polyak01}, the authors have proved $\delta_t\geq0$, then the ellipsoid $\mathcal{E}_t^{'}=\{x:(x-x_t)^TP_t^{-1}(x-x_t)\leq 1\}\supseteq\mathcal{E}_t$. As the optimization problem (\ref{Eqpre_63})-(\ref{Eqpre_64}) may lead to high computational complexity, then they advise to find the optimal value of vector $t$ by deleting $(1-\delta_t)$ in (\ref{Eqpre_63}), which yields the following optimization problem \cite{Durieu-Walter-Polyak01},
\begin{eqnarray}
\label{Eqpre_38} &&\min~~ f(P_t) \\
\label{Eqpre_39} &&~~\mbox{s.t.}~~ (t_1,\ldots,t_m)\in\mathscr{D}^+.
\end{eqnarray}
Here the vector $t=(t_1,\ldots,t_m)$ is the optimization variable.
If the optimal value $t^{*}$ of $t$ takes place within the ellipsoid $\mathcal{E}_t^{'}$, the final ellipsoidal approximation is improved by taking
\begin{eqnarray}
\label{Eqpre_57}\mathcal{E}_{t^{*}}=\{x:(x-x_{t^{*}})^T((1-\delta_{t^{*}})P_{t^{*}})^{-1}(x-x_{t^{*}})\leq 1\}.
\end{eqnarray}
Finally, we can use Algorithm \ref{alg_5} to obtain a suboptimal ellipsoid containing the intersection of these ellipsoids $\mathcal{E}_1,\ldots,\mathcal{E}_m$.
  \begin{algorithm}\label{alg_5} Bounding ellipsoid relaxation without optimizing $1-\delta_{t}$
 \begin{enumerate}
   \item Solve the optimization problem (\ref{Eqpre_38})-(\ref{Eqpre_39}) for the optimal solutions $t_i$, $i=1,\ldots,m$.
   \item Compute the shape matrix $P_t$ and center $x_t$ of the outer ellipsoid  by (\ref{Eqpre_34})-(\ref{Eqpre_35}). Meanwhile, calculate $\delta_t$ by (\ref{Eqpre_36}).
   \item   Obtain the suboptimal ellipsoid with the shape matrix $P_0=(1-\delta_{t})P_{t}$ and the center $x_0=x_t$.
 \end{enumerate}
 \end{algorithm}
 \begin{corollary}\label{cor_3}
 When the objective function is logdet or trace function, compared with Algorithm \ref{alg_5}, Algorithm \ref{alg_3} based on the decoupled SDP relaxation method can derive a tighter ellipsoid to contain the intersection.
 \end{corollary}
 \begin{proof}
 Since the optimal solution of the optimization problem (\ref{Eqpre_38})-(\ref{Eqpre_57}) is feasible for the problem (\ref{Eqpre_63})-(\ref{Eqpre_64}), according to Proposition \ref{pro_3}, we can obtain the result of Corollary \ref{cor_3}.
 \end{proof}
\begin{remark}\label{rem_3}~
 Algorithm  \ref{alg_5}  is widely applied to set-membership filter for target tracking \cite{Schweppe68}, since it has less computing time, which can be seen in Table \ref{tab_1}. However, based on Corollary \ref{cor_3}, the minimum ellipsoid calculated by Algorithm \ref{alg_3} can be tighter than that by Algorithm \ref{alg_5}. Therefore, if  Algorithm \ref{alg_3} is used in set-membership filter, it can achieve a better estimation performance (see Figs. \ref{fig4}-\ref{fig5}).
\end{remark}

In the distributed estimation fusion setting, if the cross-correlation of local sensor estimation errors is unknown or impractical, the covariance intersection (CI) algorithm \cite{Uhlmann96,Julier-Uhlmann09,Guo-Chen-Leung-Liu10} is widely used to deal with this problem.
Notice the definition of $P_t$ in (\ref{Eqpre_34}), interestingly, the steps 1-2 of Algorithm  \ref{alg_5} is same as the CI algorithm. Based on Proposition \ref{pro_3} and $0\leq\delta_t\leq1$, we can obtain the following corollary.
 \begin{corollary}\label{cor_4}
 Compared with CI algorithm,  Algorithm \ref{alg_3} based on the decoupled SDP relaxation method can derive a tighter ellipsoid containing the intersection of ellipsoids.
 \end{corollary}
Therefore, if we use the decoupled SDP relaxation method to solve the distributed estimation fusion with unavailable cross-correlation among multiple sensors, the estimation performance may be further improved (see Figs. \ref{fig6}-\ref{fig7}).

 A recursive version of the problems in (\ref{Eqpre_38})-(\ref{Eqpre_57}) is also considered by \cite{Durieu-Walter-Polyak01}. Let $\mathcal{E}^k$ be the approximate ellipsoid obtained after processing the first $k$ ellipsoids $\mathcal{E}_1,\ldots,\mathcal{E}_k$. The center and the shape matrix of ellipsoid $\mathcal{E}^k$ are denoted by $x^k$ and $P^k$. The next approximation is to find $\mathcal{E}^{k+1}$ containing $\mathcal{E}^k\bigcap\mathcal{E}_{k+1}$. The recursive algorithm is formulated in Algorithm \ref{alg_1}.
 \begin{algorithm}\label{alg_1}Recursive bounding ellipsoid relaxation without optimizing $1-\delta_{t}$
 \begin{enumerate}
   \item Initialized, $x^1=x_1$ and $P^1=P_1$.
   \item Calculate the shape matrix $P_t^{k+1}=(t(P^k)^{-1}+(1-t)P_{k+1}^{-1})^{-1}$ based on (\ref{Eqpre_34}).
   \item Compute the center $x_t^{k+1}=P_t^{k+1}(t(P^{k})^{-1}x^k+(1-t)P_{k+1}^{-1}x_{k+1})$ and $\delta_t$ based on (\ref{Eqpre_35}) and (\ref{Eqpre_36}), respectively.
   \item Minimize the objective function $f(P_t^{k+1})$ in (\ref{Eqpre_38})-(\ref{Eqpre_39}), and obtain the optimal value $t^{*}$.
   \item Achieve the shape matrix $P^{k+1}=(1-\delta_{t^{*}})P_{t^{*}}^{k+1}$ and the center $x^{k+1}=x_{t^{*}}^{k+1}$ of the ellipsoid $\mathcal{E}^{k+1}$.
   \item Go to step 2 until $k=m$.
 \end{enumerate}
 \end{algorithm}
\begin{remark}\label{rem_5}~
\begin{itemize}
  \item If the steps 4-5 are replaced by solving the optimization problem (\ref{Eqpre_50})-(\ref{Eqpre_53}), then Algorithm \ref{alg_1} may get a tighter ellipsoid to contain the intersection .
  \item Algorithm \ref{alg_1} may be applied to the distributed estimation fusion for delay systems \cite{Sun-Ma14}, but it generates a more pessimistic final ellipsoid than the nonrecursive algorithm. If we modify the shape matrix to $P^{k+1}=P_{t^{*}}^{k+1}$ in the step 5, then it is similar to the sequential CI Kalman fusion \cite{Cong-Li-Qi-Sheng16}.
\end{itemize}

\end{remark}
\section{Simulation Results} \label{sec_6}
In this section, in order to show the analytic results of the equivalence and the computational complexity, we consider two cases: static case and dynamic case. The following simulation results are under Matlab R2015a with CVX.

1) \emph{Static Case}: Since the optimization problem (\ref{Eqpre_28})-(\ref{Eqpre_29}) by the SDP relaxation method is same as (\ref{Eqpre_76})-(\ref{Eqpre_77}) by the S-procedure relaxation method, they have the same computational complexity. For the optimization problem (\ref{Eqpre_63})-(\ref{Eqpre_64}) by the bounding ellipsoid relaxation method, Durieu, Walter and Polyak \cite{Durieu-Walter-Polyak01} recognized that the direct acquisition of the optimal value of the problem (\ref{Eqpre_63})-(\ref{Eqpre_64}) becomes marginal in most examples treated so far, since it is at cost of more computation. Therefore, we focus on comparing the computational complexity of the SDP relaxation method, the decoupled SDP relaxation method and the other three algorithms.

 Suppose that there exist three local estimated ellipsoids $\mathcal{E}_1$, $\mathcal{E}_2$, $\mathcal{E}_3$, the shape matrix and the center are denoted as follows
\begin{eqnarray}
\nonumber &&P_1=\left[
                      \begin{array}{cc}
                        6 & -5 \\
                        -5 & 12 \\
                      \end{array}
                    \right],~P_2=\left[
                      \begin{array}{cc}
                        10 & 1 \\
                        1 & 3\\
                      \end{array}
                    \right],~P_3=\left[
                      \begin{array}{cc}
                        5 & 5 \\
                        5 & 9\\
                      \end{array}
                    \right]\\
\nonumber  &&x_1=[12~11]^T,~~x_2=[12~10],~~x_3=[12~\xi].
\end{eqnarray}
Here, $\xi$ is a random variable with uniform distribution in interval $[9 ~10]$. The goal is determining an ellipsoid $\mathcal{E}$ containing the intersection of the three ellipsoids. The matrix $P$ is denoted by the shape matrix ellipsoid $\mathcal{E}$. The objective function $f(P)$ is $logdet(P)$, which means the volume of the ellipsoid $\mathcal{E}$.

%Firstly, in order to get some visible results, we use two ellipsoids $\mathcal{E}_1$ and $\mathcal{E}_2$ to validate the result of Proposition \ref{pro_3}. Let $V(t)=logdet((1-\delta_t)P_t)$, based on (\ref{Eqpre_34})-(\ref{Eqpre_64}), which is the function of parameter $t$. Then the optimal value of function $V(t)$ and function $V(t)$ are plotted in Fig. \ref{alg_3}. Meanwhile, we can use Algorithm \ref{alg_3} to calculate an optimal ellipsoid to contain the intersection of $\mathcal{E}_1$ and $\mathcal{E}_2$.
%From Fig. \ref{alg_3}, it is clear to see that the optimal value by Algorithm \ref{alg_3} is same as that of function $V(t)$. Thus, the decoupled SDP relaxation method and bounding ellipsoid relaxation method have the same minimum ellipsoid to contain the intersection.

We use the above five algorithms to compute the ellipsoid $\mathcal{E}$ containing the intersection of the three ellipsoids.
Table \ref{tab_1} shows the performance of the above five algorithms. Meanwhile, the values of $\log\det(P)$ and the computing time are calculated by the average of 100 Monte Carlo runs.
It is easy to see that  Algorithm \ref{alg_2} and Algorithm \ref{alg_3} do derive the same ellipsoid to contain the intersection, which is consistent with the result of  Proposition \ref{cor_2}. But the computing time of Algorithm \ref{alg_3} is less than that of Algorithm \ref{alg_2}, the reason is that Algorithm \ref{alg_3} has less optimization variables (see Remark \ref{rem_1}).

Fig. \ref{fig1} shows that the minimum ellipsoid  derived  by Algorithm \ref{alg_2} can contain the intersection of the three ellipsoids.  Moreover, by Fig. \ref{fig2}, Table \ref{tab_1}, Remark \ref{rem_3} and Remark \ref{rem_5}, Algorithm \ref{alg_2}-\ref{alg_3} may provide the smallest volume ellipsoid compared with the other three algorithms. Although Algorithm \ref{alg_4} can produce a maximum volume ellipsoid inscribed in the intersection of the three ellipsoids, this ellipsoid is shrunk by a factor of $n$ about it center, then it is only guaranteed to contain the intersection.
If we consider the computing time, the Algorithm \ref{alg_5} might be the appropriate choice. By Table \ref{tab_1} and Remark~\ref{rem_5}, Algorithm \ref{alg_1} generates a more pessimistic ellipsoid than Algorithm \ref{alg_5}, but this recursive ellipsoid algorithm can be applied to the distributed estimation fusion for delay systems.
%These two intersecting ellipsoids are drawn on Fig. 1, which is denoted by E1 and E2, respectively. We use the three relaxation techniques to calculate the minimum ellipsoid containing the intersection of these two ellipsoids. Since we have proved that these relaxation techniques are equivalent, then the approximation minimum ellipsoids are same, and they are denoted by RE. In Fig. 1, comparing with the covariance intersection (CI) algorithm, the SDP relaxation and the other two methods can derive a tight ellipsoid to contain the intersection of these ellipsoids.

%\begin{table}[!t]\label{tab_1}
%\footnotesize
%\caption{Comparisons of different algorithms}
%\tabcolsep 10pt %space between two columns. 用于调整列间距
%\begin{tabular*}{\textwidth}{cccc}
%%\toprule
%\hline
%  Algorithm & logdet(P) & computing time \\\hline
%  Algorithm1 & 2.6293  & 2.4067\\
%  Algorithm2 & 2.6293  & 2.0996\\
%  Algorithm3 & 4.2379 & 1.3711\\
%  Algorithm4 & 2.6378  & 1.2958\\
%  Algorithm5 & 2.6500  & 2.5663\\\hline
%%\bottomrule
%\end{tabular*}
%%\caption{Comparing with different algorithm}
%\end{table}
%\begin{figure}[!t]
%\centering
%\includegraphics[width=3in]{bouding_ellipsoid.eps}
%\caption{Validate the result of Proposition \ref{pro_3}. The green \textcolor[rgb]{0.00,1.00,0.00}{$\circ$} is denoted by the optimal value of $V(t)$, which is calculated by bounding ellipsoid relaxation method. The red \textcolor[rgb]{1.00,0.00,0.00}{$*$} represents the optimal value of Algorithm \ref{alg_3}, which is computed by decoupled SDP relaxation.}
%\label{fig3}
%\end{figure}
\begin{table}[!t]
\caption{Comparisons of different algorithms}
\renewcommand{\arraystretch}{1.5}
\label{tab_1}
\centering
\begin{tabular}{|c|c|c|}
\hline
Algorithm & logdet(P) & computing time\\
\hline
Algorithm1 & 2.6293  & 2.4067\\
\hline
Algorithm2 & 2.6293  & 2.0996\\
\hline
 Algorithm3 & 4.2379 & 1.3711\\
\hline
Algorithm4 & 2.6378  & 1.2958\\
\hline
Algorithm5 & 2.6500  & 2.5663\\
\hline
\end{tabular}
\end{table}
\begin{figure}[!t]
\centering
\includegraphics[width=3in]{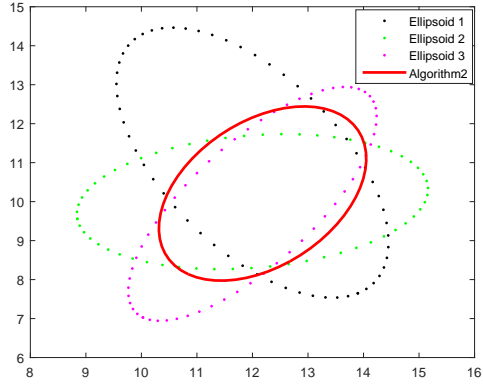}
\caption{Using Algorithm \ref{alg_3} to calculate a best outer ellipsoid containing the intersection of three ellipsoids.}
\label{fig1}
\end{figure}

2) \emph{Dynamic Case}: Consider a constant-velocity moving target with three local sensors, and the dynamic model is as follows:
\begin{eqnarray}
\nonumber x_{k+1}&=&F_kx_k+w_k\\
\nonumber y_{k+1}^i&=&x_{k+1}+v_k^i,~i=1,2,3,
\end{eqnarray}
where $F_k=\left[
                      \begin{array}{cc}
                        1 & T \\
                        0 & 1 \\
                      \end{array}
                    \right]$ and sampling time $T=1$. The process noise and measurement noise are assumed to be confined to specified ellipsoidal sets
\begin{eqnarray}
\nonumber W_k&=&\{w_k: w_k^TQ_k^{-1}w_k\leq1\}\\
\nonumber V_k&=&\{v_k: v_k^T{R_k^i}^{-1}v_k\leq1\}, i=1,2,3,
\end{eqnarray}
where
\begin{eqnarray}
\nonumber Q_k&=&\left[
                          \begin{array}{cc}
                            \frac{T^3}{3} & \frac{T^2}{2} \\
                             \frac{T^2}{2}& T \\
                          \end{array}
                        \right]\\
 \nonumber R_k^1&=&\left[
                   \begin{array}{cc}
                     20 & 0\\
                     0 & 20\\
                   \end{array}
                 \right],
  \nonumber R_k^2=\left[
                   \begin{array}{cc}
                     18 & 0\\
                     0 & 22\\
                   \end{array}
                 \right],
   \nonumber R_k^3=\left[
                   \begin{array}{cc}
                     22 & 0\\
                     0 & 18\\
                   \end{array}
                 \right].
\end{eqnarray}
The goal of this example is to find a minimum ellipsoid to contain the true state by the set-membership filter and the distributed estimation fusion, respectively. In general, the set-membership filter contains two steps: predicted step and updated step.

The predicted step of a single sensor is to compute a predicted ellipsoid $\mathcal{E}_{k+1|k}^i$ containing the true state by the state function and an updated ellipsoid $\mathcal{E}_{k|k}^i$, $i=1,2,3$. Meanwhile, the shape matrix $P_{k+1|k}^i$ and the center $x_{k+1|k}^i$ of the predicted ellipsoid are calculated as follows \cite{Durieu-Walter-Polyak01}:
\begin{eqnarray}
\nonumber P_{k+1|k}^i&=&\frac{1}{\tau_1}F_kP_{k|k}^iF_k^T+\frac{1}{\tau_2}Q_k\\
 \nonumber x_{k+1|k}^i&=&F_kx_{k|k}^i,
\end{eqnarray}
where
\begin{eqnarray}
\nonumber \tau_1^i&=&\frac{\sqrt{trace(F_kP_{k|k}^iF_k^T)}}{\sqrt{trace(F_kP_{k|k}^iF_k^T)}+\sqrt{trace(Q_k)}}\\
 \nonumber \tau_2^i&=&\frac{\sqrt{trace(Q_k)}}{\sqrt{trace(F_kP_{k|k}^iF_k^T)}+\sqrt{trace(Q_k)}},
\end{eqnarray}
$x_{k|k}^i$ and $P_{k|k}^i$ are the center and the shape matrix of the updated ellipsoid $\mathcal{E}_{k|k}^i$ at time $k$, respectively.

The updated step is to determine an updated ellipsoid $\mathcal{E}_{k+1|k+1}^i$ to contain the intersection of the predicted ellipsoid $\mathcal{E}_{k+1|k}^i$ and the measurement ellipsoid by each sensor, where the measurement ellipsoid is defined by $\mathcal{V}_{k+1}^i=\{x_{k+1}: (y_{k+1}^i-x_{k+1})^TR_{k+1}^{i^{-1}}(y_{k+1}^i-x_{k+1})\leq 1\}$. In this example, the target starts with $x_0=[1~1]$. Assume that the center and the shape matrix of the initial bounding ellipsoid are $\hat{x}_0=[2~2]^T$ and $P_0=diag(50,50)$, respectively. The following simulation results are calculated by the average of 100 Monte Carlo runs.

Firstly, we compare the estimation precision of Algorithm \ref{alg_3} with that of Algorithm \ref{alg_5} by a single sensor. Since Algorithm \ref{alg_5} is widely applied to set-membership filter in updated step, we use the measurement of sensor 1 to calculate the updated ellipsoid $\mathcal{E}_{k+1|k+1}^1$ by Algorithm \ref{alg_3} and Algorithm \ref{alg_5}, respectively. Figs. \ref{fig4}-\ref{fig5} show the root mean square error (RMSE) of the state estimation and the volume of the updated ellipsoid $\mathcal{E}_{k+1|k+1}^1$. It is clear to see that Algorithm \ref{alg_3} based on the decoupled SDP method performs better than Algorithm \ref{alg_5}, which is consistent with the result of Corollary \ref{cor_3}. Therefore, if we want to obtain the higher estimation precision, Algorithm \ref{alg_3} is a better choice.

Secondly, in the fusion center, we compare the estimation precision of Algorithm \ref{alg_3} with that of CI method. Assume that each sensor can obtain an updated ellipsoid $\mathcal{E}_{k+1|k+1}^i$ by the set-membership filter with Algorithm \ref{alg_3}, $i=1,2,3$, and these local updated ellipsoids are sent to the fusion center.
The goal of the fusion center is to calculate an optimal ellipsoid containing the intersection of $\mathcal{E}_{k+1|k+1}^i$, $i=1,2,3$. Based on Algorithm \ref{alg_3} and CI method, we can derive the fused ellipsoid, respectively. The results are plotted in Figs. \ref{fig6}-\ref{fig7}. From Figs. \ref{fig6}-\ref{fig7}, we can see that the estimation precisions of Algorithm \ref{alg_3} and CI method are higher than that of each sensor. The Algorithm \ref{alg_3} performs the best, which is also consistent with the result of Corollary \ref{cor_4}.

\begin{figure}[!t]
\centering
\includegraphics[width=3in]{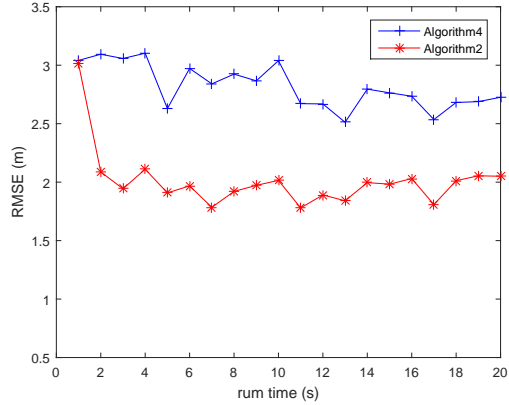}
\caption{Comparison of the RMSE of state estimation for sensor 1}
\label{fig4}
\end{figure}

\begin{figure}[!t]
\centering
\includegraphics[width=3in]{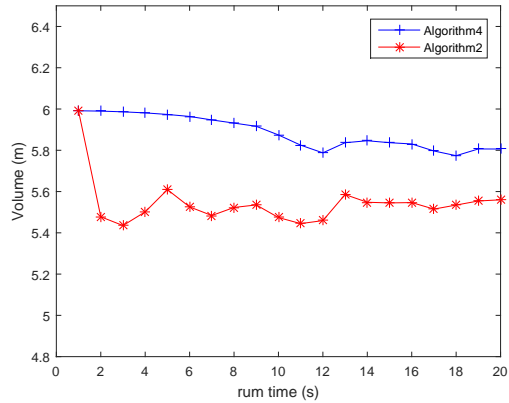}
\caption{Comparison of the volume of updated ellipsoid for sensor 1}
\label{fig5}
\end{figure}

\begin{figure}[!t]
\centering
\includegraphics[width=3in]{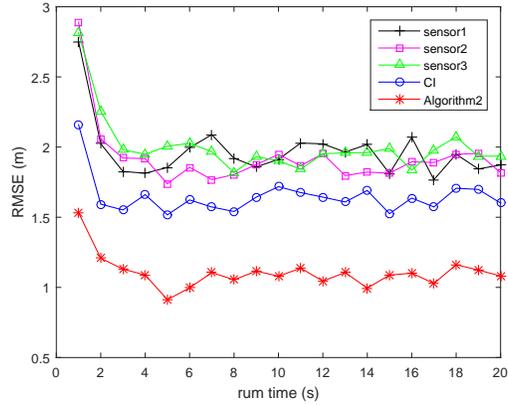}
\caption{Comparison of the RMSE of state estimation in the fusion center}
\label{fig6}
\end{figure}

\begin{figure}[!t]
\centering
\includegraphics[width=3in]{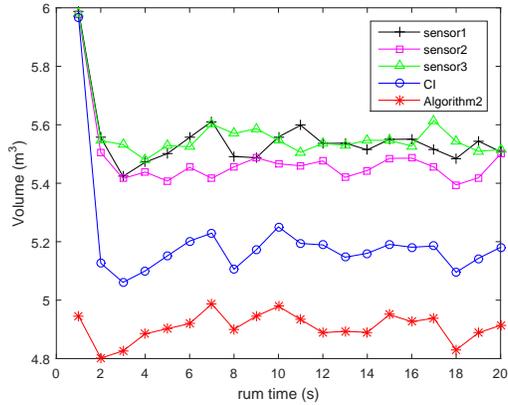}
\caption{Comparison of the volume of fused ellipsoid in the fusion center}
\label{fig7}
\end{figure}
%\begin{center}
%  \centerline
%  {\includegraphics[scale=0.6]{intersection.eps}}\vskip3mm
%\centering{\small { Fig. 2}\ \ Using Algorithm \ref{alg_3} to calculate a best outer ellipsoid containing the intersection of three ellipsoids. \label{fig1}}
%\end{center}
\section{Conclusions and Future Work}\label{sec_7}
 This paper has investigated various solving techniques for the optimization problem that determining a minimum ellipsoid containing the intersection of multiple ellipsoids. In fact, this optimization problem is difficult to solve due to that there exist the infinite number of constraints. Moreover, it is also a non-convex optimization problem, which needs to compute the maximum of the quadratic function. Therefore, there are many researchers address itself to relax this hard optimization problem to a convex optimization problem, which can be solved effectively by interior point methods.
 There are three major relaxation methods involving SDP relaxation, S-procedure relaxation and bounding ellipsoid relaxation, which are derived by different ideas or viewpoints. However, it is unclear for the interrelationships among these methods and insight to the pros and cons of these methods. This paper has revealed the equivalence among the three relaxation methods by three stages. Firstly, the SDP relaxation method can be equivalently simplified to a decoupled SDP relaxation method. Secondly, the equivalence between the SDP relaxation method and the S-procedure relaxation method can be obtained by rigorous analysis. Thirdly, we establish the equivalence between the decoupled SDP relaxation method and the bounding ellipsoid relaxation method. Therefore, the three relaxation methods are unified through the decoupled SDP relaxation method. By analysis of the computational complexity, the decoupled SDP relaxation method has the least computational burden among the three methods. The above results are helpful for the research of set-membership filter and distributed estimation fusion. Finally, the performance of each method is evaluated by some typical numerical examples in information fusion and filtering.

 We believe that this important problem, determining a minimum ellipsoid containing the intersection of multiple ellipsoids, is worth studying in future work, since it can be widely applied to many practiced fields. The future researches may involve:

\begin{itemize}
  \item Since this optimization problem is NP-hard, we may aim at finding a better relaxation technique to derive a tighter ellipsoid that containing the intersection of the ellipsoids.
  \item There are many randomized methods to deal with a lot of complex problems \cite{Huang-Palomar14,Nemirovski-Juditsky-Lan-Shapiro09}, then whether these randomized methods can be applied to this optimization problem.
  \item For large-scale problems, the interior point method may not be suitable for solving SDP. Then we hope to derive a first order method to solve these relaxed optimization problems.
\end{itemize}

\section{Appendix}
\begin{lemma}\label{lem_2}
(Schur Complements) \cite{Boyd-ElGhaoui-Feron-Balakrishnan94}: Given constant matrices $A$, $B$, $C$, where $C=C^T$ and $A=A^T$, then the condition
\begin{eqnarray}
\left[
  \begin{array}{cc}
    A & B \\
    B^T & C \\
  \end{array}
\right]\succeq0
\end{eqnarray}
is equivalent to
\begin{eqnarray}
A\succeq0, C-B^TA^{+}B\succeq0,(I-A^{+}A)B=0,
\end{eqnarray}
and also to
\begin{eqnarray}
C\succeq0, A-BC^{+}B^T\succeq0,(I-C^{+}C)B^T=0,
\end{eqnarray}
where $A^{+}$ and $C^{+}$ denote the Moore-Penrose pseudoinverse of $A$ and $C$, respectively.
\end{lemma}
\begin{proof}[Proof of Lemma \ref{the_1}]:
The Lagrangian of (\ref{Eqpre_7})-(\ref{Eqpre_8}) is
 \begin{eqnarray}
\nonumber &&L(x,X,\lambda_1,\ldots,\lambda_m,G)\\
%\nonumber &&=tr(P_0^{-1}X)-2x_0^TP_0^{-1}x+x_0^TP_0^{-1}x_0-1-tr\left(G(xx^T-X)\right)\\
%\nonumber&&-\sum_{i=1}^m\lambda_i\left(tr(P_i^{-1}X)-2x_i^TP_i^{-1}x+x_i^TP_i^{-1}x_i-1\right)\\
\nonumber&&=tr\Bigg(\Big(P_0^{-1}-\sum_{i=1}^m\lambda_iP_i^{-1}+G\Big)X\Bigg)-x^TG x\\
\nonumber&&-2x_0^TP_0^{-1}x+2\sum_{i=1}^m \lambda_ix_i^TP_i^{-1}x\\
\label{Eqpre_12}&&+x_0^TP_0^{-1}x_0-1-\sum_{i=1}^m\lambda_i(x_i^TP_i^{-1}x_i-1).
\end{eqnarray}
Then, the dual function is
 \begin{eqnarray}\label{Eqpre_13}
&&g(\lambda_1,\ldots,\lambda_m,G)=\max_{x,X}L(x,X,\lambda_1,\ldots,\lambda_m,G)\\
\nonumber&&= \Big(P_0^{-1}x_0-\sum_{i=1}^m\lambda_iP_i^{-1}x_i\Big)^T\Big(\sum_{i=1}^m\lambda_iP_i^{-1}-P_0^{-1}\Big)^{+}\\
\nonumber&&\cdot\Big(P_0^{-1}x_0-\sum_{i=1}^m\lambda_iP_i^{-1}x_i\Big)+x_0^TP_0^{-1}x_0\\
\nonumber &&-1-\sum_{i=1}^m\lambda_i(x_i^TP_i^{-1}x_i-1),
 \end{eqnarray}
 and Lagrange multipliers $\lambda_1,\ldots,\lambda_m$ need to satisfy the following constrains
  \begin{eqnarray}
 \nonumber && \lambda_i\geq0,i=1,\ldots,m\\
  \nonumber &&\sum_{i=1}^m\lambda_iP_i^{-1}-P_0^{-1}\succeq0.
   \end{eqnarray}
% \begin{eqnarray}
%\nonumber=\left\{
%     \begin{array}{ll}
%       \left(P_0^{-1}x_0-\sum_{i=1}^m\lambda_iP_i^{-1}x_i\right)^T(\sum_{i=1}^m\lambda_iP_i^{-1}-P_0^{-1})^{+}\left(P_0^{-1}x_0\right. & \lambda_i\geq0,i=1,\ldots,m\\
%       ~~\left.-\sum_{i=1}^m\lambda_iP_i^{-1}x_i\right)+x_0^TP_0^{-1}x_0-1-\sum_{i=1}^m\lambda_i(x_i^TP_i^{-1}x_i-1), & \hbox{$\sum_{i=1}^m\lambda_iP_i^{-1}-P_0^{-1}\succeq0$ }\\
%       -\infty, & \hbox{otherwise.}
%     \end{array}
%   \right.
%\end{eqnarray}
From (\ref{Eqpre_4}), (\ref{Eqpre_12})-(\ref{Eqpre_13}), we can obtain the following inequality
\begin{eqnarray}
\label{Eqpre_14}\varphi(x_0,P_0)\leq g(\lambda_1,\ldots,\lambda_m,G).
\end{eqnarray}
If $g(\lambda_1,\ldots,\lambda_m,G)\leq 0$, then $\varphi(x_0,P_0)\leq 0$, in other words, the feasible set in (\ref{Eqpre_2}) can be relaxed to the following forms based (\ref{Eqpre_13}),
\begin{eqnarray}
\nonumber &&\Big(P_0^{-1}x_0-\sum_{i=1}^m\lambda_iP_i^{-1}x_i\Big)^T\Big(\sum_{i=1}^m\lambda_iP_i^{-1}-P_0^{-1}\Big)^{+}\\
\nonumber&&\cdot\Big(P_0^{-1}x_0-\sum_{i=1}^m\lambda_iP_i^{-1}x_i\Big)+x_0^TP_0^{-1}x_0\\
\label{Eqpre_15}&&-1-\sum_{i=1}^m\lambda_i(x_i^TP_i^{-1}x_i-1)\leq0,\\
\label{Eqpre_16}&&\sum_{i=1}^m\lambda_iP_i^{-1}-P_0^{-1}\succeq0,\\
\label{Eqpre_17}&&\lambda_i\geq0,i=1,\ldots,m.
\end{eqnarray}
Using a Schur complement, we can express the (\ref{Eqpre_15})-(\ref{Eqpre_16}) as
\begin{eqnarray}
%\nonumber\left[
%  \begin{array}{cc}
%     P_0^{-1}- & -P_0^{-1}x_0+ \\
%    \sum_{i=1}^m\lambda_iP_i^{-1} & \sum_{i=1}^m\lambda_iP_i^{-1}x_i \\
%    ~& ~ \\
%    -x_0^TP_0^{-1} & x_0^TP_0^{-1}x_0-1 \\
%    +\sum_{i=1}^m\lambda_ix_i^TP_i^{-1} & -\sum_{i=1}^m\lambda_i(x_i^TP_i^{-1}x_i-1) \\
%  \end{array}
%\right]\preceq0
\nonumber\left[
  \begin{array}{cc}
    P_0^{-1}-\sum_{i=1}^m\lambda_iP_i^{-1} & -P_0^{-1}x_0+\sum_{i=1}^m\lambda_iP_i^{-1}x_i \\
    (-P_0^{-1}x_0+\sum_{i=1}^m\lambda_iP_i^{-1}x_i)^T & x_0^TP_0^{-1}x_0-1-\sum_{i=1}^m\lambda_i(x_i^TP_i^{-1}x_i-1) \\
  \end{array}
\right]\preceq0
\end{eqnarray}
or, replacing the variable $\tilde{x}_0$ by $\tilde{x}_0=P_0^{-1}x_0$,
\begin{eqnarray}
\label{Eqpre_30}\left[
                  \begin{array}{ccc}
                    P_0^{-1}-\sum_{i=1}^m\lambda_iP_i^{-1} & -\tilde{x}_0+\sum_{i=1}^m\lambda_iP_i^{-1}x_i & 0 \\
                    (-\tilde{x}_0+\sum_{i=1}^m\lambda_iP_i^{-1}x_i)^T & -1-\sum_{i=1}^m\lambda_i(x_i^TP_i^{-1}x_i-1)& \tilde{x}_0^T \\
                    0 & \tilde{x}_0 & -P_0^{-1} \\
                  \end{array}
                \right]\preceq0,
\end{eqnarray}
Combine (\ref{Eqpre_17}) and (\ref{Eqpre_30}), we can obtain the optimization problem (\ref{Eqpre_28})-(\ref{Eqpre_29}) in Lemma \ref{the_1}.
\end{proof}
%\theorem[name]{It is a test theorem.}
%\newpage
%\begin{table}[!t]
%\footnotesize
%\caption{Tabel caption}
%\label{tab1}
%\tabcolsep 49pt %space between two columns. 用于调整列间距
%\begin{tabular*}{\textwidth}{cccc}
%\toprule
%  Title a & Title b & Title c & Title d \\\hline
%  Aaa & Bbb & Ccc\footnote & Ddd\\
%  Aaa & Bbb & Ccc\footnote & Ddd\\
%  Aaa & Bbb & Ccc & Ddd\\
%\bottomrule
%\end{tabular*}
%\end{table}
%\footnotetext[1]{test1}
%\footnotetext[2]{test2}
%
%The examples at the bottom of the .tex file can help you when preparing your manuscript\footnote{test3}. We are appreciate your effort to follow our style~\cite{1,2}.
%%%%%%%%%%%%%%%%%%%%%%%%%%%%%%%%%%%%%%%%%%%%%%%%%%%%%%
%% Acknowledgements. 致谢
%%%%%%%%%%%%%%%%%%%%%%%%%%%%%%%%%%%%%%%%%%%%%%%%%%%%%%

%\section{Appendix} \label{sec_6}

%%%%%%%%%%%%%%%%%%%%%%%%%%%%%%%%%%%%%%%%%%%%%%%%%%%%%%%
%%% Supplements. 补充材料, 非必选
%%%%%%%%%%%%%%%%%%%%%%%%%%%%%%%%%%%%%%%%%%%%%%%%%%%%%%%
%\Supplements{Appendix A.}

%%%%%%%%%%%%%%%%%%%%%%%%%%%%%%%%%%%%%%%%%%%%%%%%%%%%%%%

%
%\end{thebibliography}
%\section*{References}


\begin{thebibliography}{10}
\expandafter\ifx\csname url\endcsname\relax
  \def\url#1{\texttt{#1}}\fi
\expandafter\ifx\csname urlprefix\endcsname\relax\def\urlprefix{URL }\fi
\expandafter\ifx\csname href\endcsname\relax
  \def\href#1#2{#2} \def\path#1{#1}\fi



\bibitem{Boyd-ElGhaoui-Feron-Balakrishnan94}
S.~Boyd, L.~E. Ghaoui, E.~Feron, V.~Balakrishnan, Linear Matrix Inequalities in
  System and Control Theory, Philadelphia, PA: SIAM (Studies in Applied
  Mathematics), 1994.

  \bibitem{Wang-Li12}
Y.~Wang, X.~Li, Distributed estimation fusion with unavailable
cross-correlation, IEEE Trans. Aerospace Electr. Syst. 48 (1) (2012) 259--278.

\bibitem{Shi-Chen-Lin15}
K.~Shi, H.~Chen, Y.~Lin, Probabilistic coverage based sensor scheduling for
  target tracking sensor networks, Inf. Sci. 292 (2015) 95--110.

\bibitem{Schweppe68}
F.~C. Schweppe, Recursive state estimation: Unknown but bounded errors and
  system inputs, IEEE Trans. Autom. Control 13~(1) (1968)
  22--28.

\bibitem{Shen-Zhu-Song-Luo11}
X.~Shen, Y.~Zhu, E.~Song, Y.~Luo, Minimizing {Euclidian} state estimation error
  for linear uncertain dynamic systems based on multisensor and multi-algorithm
  fusion, IEEE Trans. Inf. Theroy 57~(10) (2011) 7131--7146.

\bibitem{Sun-Tian-Lin17}
S.~Sun, T.~Tian, L.~Honglei, State estimators for systems with random parameter
  matrices, stochastic nonlinearities, fading measurements and correlated
  noises, Inf. Sci. 397--398 (2017) 118--136.

\bibitem{Polyak-Nazin-Durieu-Walter04}
B.~T. Polyak, S.~A.Nazin, C.~Durieu, E.~Walter, Ellipsoidal parameter or state
  estimation under model uncertainty, Automatica 40 (2004) 1171--1179.

\bibitem{Kwon-Lee-Park11}
O.~Kwon, S.~Lee, J.~H. Park, On the reachable set bounding of uncertain dynamic
  systems with time-varying delays and disturbances, Inf. Sci. 181
  (2011) 3735--3748.

\bibitem{Aubry-Boutayeb-Darouach08}
Y.~Becis-Aubry, M.~Boutayeb, M.~Darouach, State estimation in the presence of
  bounded disturbances, Automatica 44 (2008) 1867--1873.

\bibitem{Joachim-Deller06}
D.~Joachim, J.~John R.~Deller, Multiweight optimization in optimal bounding
  ellipsoid algorithms, IEEE Trans. Signal Process. 54~(2) (2006)
  679--690.

\bibitem{Durieu-Walter-Polyak01}
C.~Durieu, E.~Walter, B.~T. Polyak, Multi-input multi-output ellipsoidal state
  bounding, J. Optim. Theory Appl. 111~(2) (2001)
  273--303.

\bibitem{Calafiore-ElGhaoui04}
G.~Calafiore, L.~{El Ghaoui}, Ellipsoidal bounds for uncertain equations and
  dynamical systems, Automatica 40 (2004) 773--787.

\bibitem{Ros-Sabater-Thomas02}
L.~Ros, A.~Sabater, F.~Thomas, An ellipsoidal calculus based on propagation and
  fusion, IEEE Trans. Syst, Man, Cybern. Part B Cybern. 32~(4) (2002) 430--442.

\bibitem{Duan-Li11}
Z.~Duan, X.~Li, Lossless linear transformation of sensor data for
distributed estimation fusion, IEEE Trans. Signal Process. 59~(1) (2011)
  362--372.

\bibitem{Yang-Liang-Pan-Qin-Yang16}
Y.~Yang, Y.~Liang, Q.~Pan, Y.~Qin, F.~Yang, Distributed fusion estimation with
  square-root array implementation for markovian jump linear systems with
  random parameter matrices and cross-correlated noises, Inf. Sci.
  370-371 (2016) 446--462.

\bibitem{Yuan-Wang-Guo17}
Y.~Yuan, Z.~Wang, L.~Guo, Distributed quantized multi-modal ${H}_{\infty}$
  fusion filtering for two-time-scale systems, Inf. Sci. (2017) In
  Press.

\bibitem{Guo-Chen-Leung-Liu10}
Q.~Guo, S.~Chen, H.~Leung, S.~Liu, Covariance intersection based image fusion
  technique with application to pansharpening in remote sensing, Inf. Sci. 180 (2010) 3434--3443.

\bibitem{Uhlmann96}
J.~K. Uhlmann, General data fusion for estimates with unknown cross
  covariances, in: Proceedings of the SPIE aerosense conference 2755~(14) (1996)
  536--547.

\bibitem{Julier-Uhlmann09}
S.~Julier, J.~K. Uhlmann, General decentralized data fusion with covariance
  intersection {(CI)}, Handbook of multisensor data fusion: theory and
  practice, New York: CRC Press, 2009.

\bibitem{Cong-Li-Qi-Sheng16}
J.~Cong, Y.~Li, G.~Qi, A.~Sheng, An order insensitive sequential fast
  covariance intersection fusion algorithm, Inf. Sci. 367--368
  (2016) 28--40.

\bibitem{Ben-Tal-Nemirovski98}
A.~Ben-Tal, A.~Nemirovski, Robust convex optimization, Math. Oper. Res. 23 (1998) 769--805.

\bibitem{Henrion-Tarbouriech-Arzelier01}
D.~Henrion, S.~Tarbouriech, D.~Arzelier, {LMI} approximations for the radius of
  the intersection of ellipsoids: Survey, J. Optim. Theory Appl.
   108~(1) (2001) 1--28.

\bibitem{Palomar-Eldar10}
D.~P. Palomar, Y.~C. Eldar, Convex optimization in signal processing and
  communications, Cambridge university press, 2010.

\bibitem{Luo-Ma-So-Ye-Zhang10}
Z.~Luo, W.~Ma, A.~M.-C. So, Y.~Ye, S.~Zhang, Semidefinite relaxation of
  quadratic optimization problems, IEEE Signal Process. Mag. 27~(3)
  (2010) 20--34.

\bibitem{Boyd-Vandenberghe04}
S.~Boyd, L.~Vandenberghe, Convex optimization, Cambridge University Press,
  2004.

\bibitem{Nesterov-Nemirovski94}
Y.Nesterov, A.Nemirovski, Interior point polynomial methods in convex
  programming: Theroy and applications, Philadelphia, PA: SIAM.

\bibitem{Nemirovski-Roos-Terlaky99}
A.~Nemirovski, C.~Roos, T.~Terlaky, On maximization of quadratic form over
  intersection of ellipsoids with common center, Math. Program.
  86~(3) (1999) 463--473.

\bibitem{Eldar-Beck-Teboulle08}
Y.~C. Eldar, A.~Beck, M.~Teboulle, A minimax chebyshev estimator for bounded
  error estimation, IEEE Trans. Signal Process. 56~(4) (2008)
  1388--1397.



\bibitem{Noack-Sijs-Reinhardt-Hanebeck17}
B.~Noack, J.~Sijs, M.~Reinhardt, U.~D. Hanebeck, Decentralized data fusion with
  inverse covariance intersection, Automatica 79 (2017) 35--41.


\bibitem{Horst-Tuy96}
R.~Horst, H.~Tuy, Global optimization: deterministic approaches, 3rd Edition,
  Springer Verlag, Germany, 1996.

\bibitem{Vandenberghe-Boyd96}
L.Vandenberghe, S.Boyd, Semidefinite programming, SIAM Rev. 38~(1) (1996)
  49--95.

\bibitem{Grant-Boyd-Ye09}
M.~Grant, S.~Boyd, Y.~Ye, {CVX}: {Matlab} software for disciplined convex
  programming [online], Available: http://www.stanford.edu/~boyd/cvx.

\bibitem{Yakubovich71}
V.~A. Yakubovich, {S-procedure} in nonlinear control theory, Vestnik Leningrad
  University (English translation in Vestnik Leningrad University, pp. 73--93,
  1977) (1971) 62--77.

\bibitem{Hanif-Tran-Antti-Glisic13}
M.~F. Hanif, Le-NamTran, A.~T$\ddot{o}$lli, S.~Glisic, Efficient solutions for
  weighted sum rate maximization in multicellular networks with channel
  uncertainties, IEEE Trans. Signal Process. 61~(22) (2013) 5659
  --5674.

\bibitem{Sun-Ma14}
S.~Sun, J.~Ma, Linear estimation for networked control systems with random
  transmission delays and packet dropouts, Inf. Sci. 269 (2014)
  349--365.

\bibitem{Huang-Palomar14}
Y.~Huang, D.~P. Palomar, Randomized algorithms for optimal solutions of
  double-sided {QCQP} with applications in signal processing, IEEE Trans. Signal Process. 62~(5) (2014) 1093--1108.

\bibitem{Nemirovski-Juditsky-Lan-Shapiro09}
A.~Nemirovski, A.~Juditsky, G.~Lan, A.~Shapiro, Robust stochastic approximation
  approach to stochastic programming, SIAM J. Optimi. 19~(4)
  (2009) 1574--1609.

\end{thebibliography}
\end{document}